\documentclass[a4paper,oneside,reqno,final]{amsart}
\usepackage[utf8]{inputenc}
\usepackage[a4paper,margin=3cm]{geometry} 
\usepackage{bm}
\usepackage{amsmath}
\usepackage{amsthm}
\usepackage{amscd}
\usepackage{amssymb}
\usepackage{MnSymbol}
\usepackage{latexsym}
\usepackage{eucal}
\usepackage{dsfont}
\usepackage{mathtools}
\usepackage{enumitem}
\usepackage[notref,notcite]{showkeys}
\usepackage{todonotes}

\usepackage{verbatim}
\usepackage{graphicx}
\usepackage{float}
\usepackage{array,booktabs}

\usepackage[colorlinks,pdftex]{hyperref}
\hypersetup{
linkcolor=black,
citecolor=black,
pdftitle={},
pdfauthor={Franziska K\"uhn},
pdfkeywords={},
}

\swapnumbers
\theoremstyle{theorem} \newtheorem{thm}{Theorem}[section]
\theoremstyle{theorem} \newtheorem{lem}[thm]{Lemma}
\theoremstyle{theorem} \newtheorem{prop}[thm]{Proposition}
\theoremstyle{theorem} \newtheorem{kor}[thm]{Corollary}
\theoremstyle{definition} 
\theoremstyle{remark} 
\theoremstyle{remark} 
\theoremstyle{definition} 
\theoremstyle{remark} \newtheorem{bem}[thm]{Remark}
\theoremstyle{remark} 
\theoremstyle{definition}  \newtheorem{bsp}[thm]{Example}
\theoremstyle{definition}  
\theoremstyle{definition}

\DeclareMathOperator \re {Re}
\DeclareMathOperator \im {Im}

\DeclareMathOperator \spt {supp}

\DeclareMathOperator \tr {tr}

\newcommand{\I}{\mathds{1}}
\newcommand\floor[1]{\left\lfloor #1 \right\rfloor}
\newcommand\fa{\qquad \text{for all \ }}

\newcommand\mc[1] {\mathcal{#1}}
\newcommand\mbb[1] {\mathds{#1}}

\newcommand{\eps}{\varepsilon}

\begin{document}

\title[Interior Schauder estimates for equations associated with L\'evy operators]{Interior Schauder estimates for elliptic equations associated with L\'evy operators}
\author[F.~K\"{u}hn]{Franziska K\"{u}hn} 
\address[F.~K\"{u}hn]{TU Dresden, Fachrichtung Mathematik, Institut f\"{u}r Mathematische Stochastik, 01062 Dresden, Germany.}
\email{franziska.kuehn1@tu-dresden.de}

\begin{abstract}
	We study the local regularity of solutions $f$ to the integro-differential equation $$ Af=g \quad \text{in $U$}$$ associated with the infinitesimal generator $A$ of a L\'evy process $(X_t)_{t \geq 0}$. Under the assumption that the transition density of $(X_t)_{t \geq 0}$ satisfies a certain gradient estimate, we establish interior Schauder estimates for both pointwise and weak solutions $f$. Our results apply for a wide class of L\'evy generators, including generators of stable L\'evy processes and subordinated Brownian motions.
\end{abstract}
\maketitle

\section{Introduction}

Let $(X_t)_{t \geq 0}$ be a $d$-dimensional L\'evy process. By the L\'evy--Khintchine formula, $(X_t)_{t \geq 0}$ is uniquely characterized (in distribution) by its infinitesimal generator $A$, which is an integro-differential operator with representation
\begin{equation*}
	Af(x) = b \cdot \nabla f(x) + \frac{1}{2} \tr(Q \cdot \nabla^2 f(x)) + \int_{\mbb{R}^d \backslash \{0\}} \left( f(x+y)-f(x)-\nabla f(x) \cdot y \I_{(0,1)}(|y|) \right) \, \nu(dy)
\end{equation*}
for $f \in C_c^{\infty}(\mbb{R}^d)$, here $(b,Q,\nu)$ denotes the L\'evy triplet of $(X_t)_{t \geq 0}$, cf.\ Section~\ref{def}. In this paper, we study the local H\"older regularity of weak and pointwise solutions $f$ to the integro-differential equation \begin{equation*}
	Af = g \quad \text{in $U$}
\end{equation*}
for open sets $U \subseteq \mbb{R}^d$. We are interested in interior Schauder estimates, i.e.\ our aim is to describe the regularity of $f$ on the set $\{x \in U; d(x,U^c)>\delta\}$ for $\delta>0$ and to establish estimates for its H\"older norm. We will see that there is a close connection between the regularity of $f$ and volatility of the L\'evy process: the higher the volatility (caused by a non-vanishing diffusion component or a high small-jump activity), the higher the regularity of $f$. \par
For the particular case that there is no jump part, i.e.\ $\nu=0$, the generator $A$ is a second-order differential operator and interior Schauder estimates for solutions to $Af=g$ are well studied, see e.g.\ Gilbarg \cite{gilbarg}. One of the most prominent non-local L\'evy operators is the fractional Laplacian $-(-\Delta)^{\alpha/2}$, $\alpha \in (0,2)$, defined by
\begin{equation*}
	-(-\Delta)^{\alpha/2} f(x) = c_{d,\alpha} \int_{y \neq 0} \left( f(x+y)-f(x)-\nabla f(x) \cdot y \I_{(0,1)}(|y|) \right) \frac{1}{|y|^{d+\alpha}} \, dy
\end{equation*}
for some normalizing constant $c_{d,\alpha}>0$. The fractional Laplacian is the infinitesimal generator of the isotropic $\alpha$-stable L\'evy process and plays an important role in analysis and probability theory, see e.g.\ the survey paper \cite{kwas17} for a detailed discussion. Global regularity estimates for solutions to $-(-\Delta)^{\alpha/2} f=g$ go back to Stein \cite{stein}, see also Bass \cite{bass08}. Since then, several extensions and refinements of these estimates have been obtained. Ros-Oton \& Serra \cite{ros-oton16} studied the interior H\"older regularity of solutions to equations $Af=g$ associated with symmetric $\alpha$-stable operators and established under a mild degeneracy condition on the spectral measure estimates of the form \begin{equation*}
	\|f\|_{C^{\alpha+\kappa}(B(0,1))} \leq c \left(\|f\|_{C^{\kappa}(\mbb{R}^d)}+\|g\|_{C^{\kappa}(B(0,2))}\right)
\end{equation*}
for $\kappa \geq 0$ such that $\alpha+\kappa$ is not an integer. In the recent paper \cite{reg-levy}, global Schauder estimates \begin{equation}
	\|f\|_{\mc{C}_b^{\alpha+\kappa}(\mbb{R}^d)} \leq c \left(\|f\|_{\infty} + \|g\|_{\mc{C}_b^{\kappa}(\mbb{R}^d)}\right), \quad \kappa \geq 0, \label{intro-eq1}
\end{equation}
were obtained for a wide class of L\'evy processes satisfying a certain gradient estimate, see \eqref{A2} below; here $\mc{C}_b^{\beta}(\mbb{R}^d)$ denotes the H\"older--Zygmund space of order $\beta$, cf.\ Section~\ref{def} for the definition. Moreover, there are numerous results on the regularity of functions which are harmonic with respect to a L\'evy generator, see e.g.\ \cite{kwas18,hansen17,liouville,ryznar15,szt10}. Let us mention that the regularity of solutions integro-differential equations $Af=g$ has been studied, more generally, for classes of L\'evy-type operators, see e.g.\ \cite{bass08,dong12,imbert,jac2,reg-feller,lunardi98,priola18}. Of course, Schauder estimates are also of interest for parabolic equations, we point the interested reader to the recent works \cite{raynal19,hao20,mik} and the references therein.  \par
In this paper, we combine the global Schauder estimates from \cite{reg-levy}, cf.\ \eqref{intro-eq1}, with a truncation technique to derive local H\"older estimates for solutions to $Af=g$.  We will assume that the L\'evy process $(X_t)_{t \geq 0}$ with L\'evy triplet $(b,Q,\nu)$ satisfies the following conditions.
	\begin{enumerate}[label*=\upshape (C\arabic*),ref=\upshape C\arabic*] 
		\item\label{A1} The characteristic exponent $\psi$ satisfies the Hartman--Wintner growth condition \begin{equation*}
			\lim_{|\xi| \to \infty} \frac{|\re \psi(\xi)|}{\log(1+|\xi|)} = \infty.
		\end{equation*}
		\item\label{A2} There exist constants $M>0$ and $\alpha >0$ such that the transition density $p_t$, $t>0$, satisfies the gradient estimate \begin{equation*}
			\int_{\mathbb{R}^d} |\nabla p_t(x)| \, dx \leq M t^{-1/\alpha}, \qquad t \in (0,1).
		\end{equation*}
		\item\label{A3} Either $\alpha>1$ or $Q=0$. Moreover, $\alpha+1 > \gamma$ for a constant $\gamma \in (0,2]$ with $\int_{|y| \leq 1} |y|^{\gamma} \, \nu(dy)<\infty$.
\end{enumerate}
It follows from \eqref{A1},\eqref{A2} that the global Schauder estimate \eqref{intro-eq1} holds, and we will use \eqref{A3} to localize these estimates. Before stating our results, let us give some remarks on  \eqref{A1}-\eqref{A3}. 

\begin{bem} \label{main-0} \begin{enumerate}
	\item If $\alpha>1$, we may choose $\gamma=2$ in \eqref{A3} since $\int_{|y| \leq 1} |y|^{2} \, \nu(dy)<\infty$ holds for any L\'evy measure. 
	\item\label{main-0-i} The Hartman--Wintner condition \eqref{A1} implies that the law of $X_t$, $t>0$, has a density $p_t \in C_b^{\infty}(\mbb{R}^d)$ with respect to Lebesgue measure, see \cite{knop13} for a detailed discussion.
	\item\label{main-0-ii} The constant $\alpha$ in \eqref{A2} is always less or equal than $2$. This follows from the fact that the growth condition $|\psi(\xi)| \leq c(1+|\xi|^{\varrho})$, $\xi \in \mbb{R}^d$, implies $\alpha \leq \varrho$, cf.\ \cite[Remark 3.2(ii)]{reg-levy}. Roughly speaking, $\alpha$ is a measure for the volatility of the process: if $\alpha>0$ is small, then the volatility is small (i.e.\ vanishing diffusion part and few small jumps) and, conversely, if $\alpha$ is close to $2$, then $(X_t)_{t \geq 0}$ has a high volatility (many small jumps or non-vanishing diffusion part).
	\item\label{main-0-iii} If the diffusion matrix $Q$ is positive definite, then \eqref{A1}-\eqref{A3} hold with $\gamma=\alpha=2$, cf.\ Example~\ref{ex-3}. Further classes of L\'evy processes satisfying \eqref{A1}-\eqref{A3} will be presented in Section~\ref{ex}. 
	\item\label{main-0-iv} Condition \eqref{A3} is essentially a balance condition on the growth of $\re \psi$ at infinity. If we set $\beta = \max\{2 \I_{Q \neq 0}, \gamma\}$, then $|\re \psi(\xi)| \leq C(1+|\xi|^{\beta})$. On the other hand, \eqref{A2} means, roughly, that $\re \psi(\xi)\geq c (1+|\xi|^{\alpha})$. Since \eqref{A3} is equivalent to $|\beta-\alpha|<1$, this shows that the lower and upper growth rate of $\re \psi$ at infinity should be sufficiently close to each other, e.g.\ $\psi(\xi,\eta) := |\xi|^{\alpha} + |\eta|^{\beta}$ does not satisfy \eqref{A3} is $|\beta-\alpha|>1$, cf.\ Example~\ref{ex-7}.
\end{enumerate} \end{bem}

Next we state our main results; see Section~\ref{def} for the definition of the notation used in the statements.

\begin{thm} \label{main-1}
	Let $(X_t)_{t \geq 0}$ be a L\'evy process with infinitesimal generator $(A,\mc{D}(A))$ satisfying \eqref{A1}-\eqref{A3}, and denote by $\alpha \in (0,2]$ the constant from \eqref{A2}. Let $f$ be a weak solution to the equation \begin{equation*}
			Af = g \, \, \text{in $U$}
		\end{equation*}
		for an open set $U \subseteq \mbb{R}^d$. Set $U_{\delta} := \{x \in U; d(x,U^c) > \delta\}$ for $\delta>0$.
	\begin{enumerate}
	\item\label{main-1-i} If $f \in L^{\infty}(\mbb{R}^d)$ and $g \in L^{\infty}(U)$, then $f$ has a modification $\tilde{f}$ which is continuous on $U$ and satisfies the interior Schauder estimate \begin{equation}
		\|\tilde{f}\|_{\mc{C}_b^{\alpha}(U_{\delta})} \leq C_{\delta} \left(\|f\|_{L^{\infty}(\mbb{R}^d)} + \|g\|_{L^{\infty}(U)}\right) \label{main-eq11}
	\end{equation}
	for every $\delta>0$. The constant $C_{\delta}$ does not depend on $f$, $g$. 
	\item\label{main-1-ii} If $f \in \mc{C}_b^{\kappa}(\mbb{R}^d)$ and $g \in \mc{C}_b^{\kappa}(U)$ for some $\kappa >0$, then there exists for every $\delta>0$ a constant $C_{\delta}>0$ (independent of $f$, $g$) such that \begin{equation}
		\|f\|_{\mc{C}_b^{\kappa+\alpha}(U_{\delta})} \leq C_{\delta} \left(\|f\|_{\mc{C}_b^{\kappa}(\mbb{R}^d)} + \|g\|_{\mc{C}_b^{\kappa}(U)}\right). \label{main-eq13}
	\end{equation}
\end{enumerate} \end{thm}

Theorem~\ref{main-1} applies for a wide class of L\'evy processes. It generalizes the interior Schauder estimates for stable processes obtained in \cite{ros-oton16}, and for the particular case that there is no jump part, i.e.\ $\nu=0$, we recover the classical regularity estimates for second-order differential operators with constant coefficients, see Section~\ref{ex} for details and further examples.

\begin{bem} \label{main-2} \begin{enumerate}
	\item\label{main-2-1} The weak solution $f \in L^{\infty}(\mbb{R}^d)$ to $Af=g$ is only determined up to a Lebesgue null set. The interior Schauder estimate \eqref{main-eq11} implies continuity of $\tilde{f}$ on $U$, and so \eqref{main-eq11} cannot hold for any representative $\tilde{f}$ of $f$ but only for a suitably chosen representative.
	\item\label{main-2-v} In the proof of Theorem~\ref{main-2} we use \eqref{A1},\eqref{A2} only to obtain from \cite[Theorem 1.1]{reg-levy} the global Schauder estimate \begin{equation}
		\|f\|_{\mc{C}_b^{\alpha}(\mbb{R}^d)} \leq c \left( \|f\|_{\infty} + \|Af\|_{\infty} \right), \qquad f \in C_b^2(\mbb{R}^d). \label{main-eq15}
	\end{equation}
	Consequently, the interior Schauder estimates in Theorem~\ref{main-2} hold for any L\'evy process $(X_t)_{t \geq 0}$ with generator $(A,\mc{D}(A))$ satisfying the global Schauder estimate \eqref{main-eq15} and the balance condition \eqref{A3}. This means that (the proof of) Theorem~\ref{main-2} actually gives a general procedure to localize Schauder estimates.
	\item\label{main-2-iv} If we interpret the constant $\alpha$ from \eqref{A2} as a measure for the volatility of $(X_t)_{t \geq 0}$, cf.\ Remark~\ref{main-0}\eqref{main-0-ii}, then Theorem~\ref{main-1} shows that a high volatility of the L\'evy process $(X_t)_{t \geq 0}$ results in a high regularity of $f|_U$. This is a natural result, and we believe the regularity estimates to be optimal for many L\'evy processes. In some cases, certain properties of the L\'evy process or the L\'evy triplet may lead to an additional smoothing effect; for instance, Grzywny \& Kwa\'snicki \cite[Theorem 1.7]{kwas18} studied the regularity of harmonic functions $f$ (i.e.\ $Af=0$) associated with unimodal L\'evy processes and showed that the regularity of the density of the L\'evy measure $\nu$ carries over to $f$; this regularity of $f$ is not related to the volatility of the process.
	\item\label{main-2-ii} In general, the assumption $f \in \mc{C}_b^{\kappa}(\mbb{R}^d)$ in Theorem~\ref{main-1}\eqref{main-1-ii} cannot be relaxed to $f \in \mc{C}_b^{\kappa}(U)$; for stable processes a counterexample can be found in \cite[Proposition 6.1]{ros-oton16}.
	\item\label{main-2-iii} If $f$ is in the domain of the strong infinitesimal generator of $(X_t)_{t \geq 0}$, then \cite[Theorem 1.1]{reg-levy} gives $f \in \mc{C}_b^{\alpha}(\mbb{R}^d)$, and so the assumption $f \in \mc{C}_b^{\kappa}(\mbb{R}^d)$ in Theorem~\ref{main-1}\eqref{main-1-ii} is automatically satisfied for $\kappa \leq \alpha$, see also Corollary~\ref{main-5} below.
\end{enumerate} \end{bem}

Our second main result gives interior Schauder estimates for pointwise solutions to the equation $Af=g$.

\begin{kor} \label{main-3}
	Let $(X_t)_{t \geq 0}$ be a L\'evy process satisfying \eqref{A1}-\eqref{A3}, and let $U \subseteq \mbb{R}^d$ be an open set. Let $f \in \mc{B}_b(\mbb{R}^d)$ be a function such that \begin{equation*}
		 g(x) := \lim_{t \to 0} \frac{\mbb{E}f(x+X_t)-f(x)}{t}
	\end{equation*}
	exists for all $x \in U$ and assume that \begin{equation}
		\sup_{x \in K} \sup_{t \in (0,1)} \left| \frac{\mbb{E}f(x+X_t)-f(x)}{t} \right| < \infty \label{main-eq23}
	\end{equation}
	for any compact set $K \subseteq U$. Denote by $\alpha \in (0,2]$ the constant from \eqref{A2}. \begin{enumerate}
		\item\label{main-3-i} If $g \in L^{\infty}(U)$ then there exists for any $\delta>0$ a constant $C_{\delta}>0$ (independent of $f$, $g$) such that
		\begin{equation*}
			\|f\|_{\mc{C}_b^{\alpha}(U_{\delta})} \leq C_{\delta} \left(\|f\|_{L^{\infty}(\mbb{R}^d)} + \|g\|_{L^{\infty}(U)}\right).
		\end{equation*}
		\item\label{main-3-ii}  If $f \in \mc{C}_b^{\kappa}(\mbb{R}^d)$ and $g \in \mc{C}_b^{\kappa}(U)$ for some $\kappa >0$ then there exists for any $\delta>0$ a constant $C_{\delta}>0$ (independent of $f$, $g$) such that \begin{equation*}
		\|f\|_{\mc{C}_b^{\kappa+\alpha}(U_{\delta})} \leq C_{\delta} \left(\|f\|_{\mc{C}_b^{\kappa}(\mbb{R}^d)} + \|g\|_{\mc{C}_b^{\kappa}(U)}\right).
		\end{equation*}
	\end{enumerate}
\end{kor}

As an immediate consequence, we obtain local Schauder estimates for functions in the domain in the strong infinitesimal generator. They extend in a natural way the global Schauder estimates from \cite{reg-levy}.

\begin{kor} \label{main-5}
	Let $(X_t)_{t \geq 0}$ be a L\'evy process satisfying \eqref{A1}-\eqref{A3}, and denote by $\alpha \in (0,2]$ the constant from \eqref{A2}. Let $f$ be a function in the domain of the strong infinitesimal generator, i.e.\ $f$ is a continuous function vanishing at infinity and the limit \begin{equation*}
		Af(x) := \lim_{t \to 0} \frac{\mbb{E} f(x+X_t)-f(x)}{t}
	\end{equation*}
	exists uniformly in $x \in \mbb{R}^d$. \begin{enumerate}
		\item\label{main-5-i} For each $\delta>0$ there exists a finite constant $C_{\delta}$ (not depending on $f$) such that \begin{equation*}
			\|f\|_{\mc{C}_b^{\alpha}(B(x,\delta))} \leq C_{\delta} \left(\|f\|_{\infty} + \|Af\|_{\infty,B(x,2\delta)} \right)
		\end{equation*}
		for all $x \in \mbb{R}^d$.
		\item\label{main-5-ii} If $Af \in \mc{C}_b^{\kappa}(\mbb{R}^d)$ for some $\kappa>0$, then $f \in \mc{C}_b^{\kappa+\alpha}(\mbb{R}^d)$ and \begin{equation*}
			\|f\|_{\mc{C}_b^{\alpha+\kappa}(B(x,\delta))} \leq C_{\delta} \left( \|f\|_{\infty} + \|Af\|_{\infty}+ \|Af\|_{\mc{C}_b^{\kappa}(B(x,2\delta))}  \right), \quad x \in \mbb{R}^d,\, \delta>0,
		\end{equation*}
		for a finite constant $C_{\delta}$, which does not depend on $f$, $g$.
	\end{enumerate}
\end{kor}

The proof of Corollary~\ref{main-5} shows that the interior Schauder estimates \eqref{main-5-i},\eqref{main-5-ii} actually hold for any L\'evy process $(X_t)_{t \geq 0}$ with generator $(A,\mc{D}(A))$ satisfying \eqref{A3} and the global Schauder estimate 
\begin{equation*}
		\|f\|_{\mc{C}_b^{\alpha}(\mbb{R}^d)} \leq c \left( \|f\|_{\infty} + \|Af\|_{\infty} \right), \qquad f \in C_b^2(\mbb{R}^d),
\end{equation*}
see Remark~\ref{main-2}\eqref{main-2-v}. \par \medskip

This paper is organized as follows. In Section~\ref{def} we introduce basic definitions and notation. Our main results are proved in Section~\ref{p}, and examples are presented in Section~\ref{ex}.

\section{Definitions} \label{def}

We consider the $d$-dimensional Euclidean space $\mbb{R}^d$ with the canonical scalar product $x \cdot y := \sum_{j=1}^d x_j y_j$ and the Borel $\sigma$-algebra $\mc{B}(\mbb{R}^d)$ generated by the open balls $B(x,r)$. If $f$ is a real-valued function, then $\spt f$ denotes its support, $\nabla f$ the gradient and $\nabla^2 f$ the Hessian of $f$. For $\alpha \geq 0$ we set \begin{equation*}
	\lfloor \alpha \rfloor := \max\{k \in \mbb{N}_0; k \leq \alpha\}. 
\end{equation*}

\emph{Function spaces:} $\mc{B}_b(\mbb{R}^d)$ is the space of bounded Borel measurable functions $f: \mbb{R}^d \to \mbb{R}$. The smooth functions with compact support are denoted by $C_c^{\infty}(\mbb{R}^d)$. Superscripts $k\in\mbb{N}$ are used to denote the order of differentiability, e.g.\ $f \in C_{b}^k(\mbb{R}^d)$ means that $f$ and its derivatives up to order $k$ are bounded continuous functions. For $U \subseteq \mbb{R}^d$ we set \begin{equation*}
	\|f\|_{\infty,U} := \sup_{x \in U} |f(x)| \quad \text{and} \quad \|f\|_{\infty} := \|f\|_{\infty,\mbb{R}^d}.
\end{equation*}
For every $\alpha \geq 0$ and every open set $U \subseteq \mbb{R}^d$ we define the \emph{H\"older--Zygmund space} $\mc{C}_b^{\alpha}(U)$ by 
\begin{equation*}
	\mc{C}_b^{\alpha}(U) := \left\{f \in C_b(U); \|f\|_{\mc{C}_b^{\alpha}(U)} := \sup_{x \in U} |f(x)|+ \sup_{x \in U} \sup_{0<|h| < \min\{1,r_x/k\}} \frac{|\Delta_h^k f(x)|}{|h|^{\alpha}} < \infty \right\}
\end{equation*}
where $r_x := d(x,U^c)$ is the distance of $x$ from the complement of $U$, $k \in \mbb{N}$ is the smallest natural number strictly larger than $\alpha$ and
\begin{equation*}
	\Delta_h f(x) := f(x+h) - f(x) \qquad \Delta_h^{\ell} f(x) := \Delta_h (\Delta_h^{\ell-1} f)(x), \quad \ell \geq 2
\end{equation*}
are \emph{iterated difference operators}. For $U=\mbb{R}^d$ it is known that replacing $k$ by an arbitrary number $j$ strictly larger than $\alpha$ gives an equivalent norm, i.e. \begin{equation}
	\|f\|_{\mc{C}_b^{\alpha}(\mbb{R}^d)} =\|f\|_{\infty} + \sup_{x \in \mbb{R}^d} \sup_{0<|h| < 1} \frac{|\Delta_h^k f(x)|}{|h|^{\alpha}} \asymp \|f\|_{\infty} +  \sup_{x \in \mbb{R}^d} \sup_{0<|h| < 1} \frac{|\Delta_h^{j} f(x)|}{|h|^{\alpha}}, \label{def-eq5}
\end{equation}
cf.\ \cite[Theorem 2.7.2.2]{triebel78}. We need a localized version of this result.
\begin{lem} \label{def-11}
	Let $\alpha \in (0,\infty)$, and let $U \subseteq \mbb{R}^d$ be open. The following statements hold for any $j \geq k := \lfloor \alpha \rfloor+1$: \begin{enumerate}
		\item\label{def-11-i} There exists a constant $c>0$ such that \begin{equation*}
			\sup_{0<|h| \leq r} \frac{|\Delta_h^k f(x)|}{|h|^{\alpha}} 
			\leq c r^{-\alpha} \|f\|_{\infty,U} + c \sup_{0<|h| \leq r/j} \sup_{z \in B(x,r(k+1))} \frac{|\Delta_h^j f(z)|}{|h|^{\alpha}} 
		\end{equation*}
		for all $f \in C_b(U)$, $r >0$, and $x \in U$ with $B(x,r(k+2)) \subseteq U$.
		\item\label{def-11-ii} If $\alpha>1$ then there exists a constant $c>0$ such that \begin{equation*}
			|\nabla f(x)|
			\leq c r^{-\alpha} \|f\|_{\infty,U} + c \sup_{0<|h| \leq r/j} \sup_{z \in B(x,r(k+1))} \frac{|\Delta_h^j f(z)|}{|h|^{\alpha}} 
		\end{equation*}
		for all $f \in C_b^1(U)$, $r >0$ and $x \in U$ with $B(x,r(k+2)) \subseteq U$.
	\end{enumerate}
\end{lem}

We defer the proof of Lemma~\ref{def-11} to the appendix. For $\alpha \in (0,\infty) \backslash \mbb{N}$ the H\"older--Zygmund space $\mc{C}_b^{\alpha}(\mbb{R}^d)$ coincides with the ``classical'' H\"older space $C_b^{\alpha}(\mbb{R}^d)$, cf.\ \cite[Theorem 2.7.2.1]{triebel78}. If $\alpha=k \in \mbb{N}$, then the inclusion $C_b^{k}(\mbb{R}^d) \subseteq \mc{C}_b^k(\mbb{R}^d)$ is strict.

\emph{L\'evy processes:} Let $(\Omega,\mc{A},\mbb{P})$ be a probability space. A stochastic process $X_t: \Omega \to \mbb{R}^d$, $t \geq 0$, is a ($d$-dimensional) \emph{L\'evy process} if $X_0=0$ almost surely, $(X_t)_{t \geq 0}$ has independent and stationary increments and $t \mapsto X_t(\omega)$ is right-continuous with finite left-hand limits for almost all $\omega \in \Omega$. The L\'evy--Khintchine formula shows that every L\'evy process is uniquely determined in distribution by its \emph{characteristic exponent} $\psi: \mbb{R}^d \to \mbb{C}$ satisfying \begin{equation*}
	\mbb{E}\exp(i \xi \cdot X_t) = \exp(-t \psi(\xi)), \qquad t \geq 0, \, \xi \in \mbb{R}^d.
\end{equation*}
The characteristic exponent $\psi$ has a \emph{L\'evy--Khintchine representation} \begin{equation*}
	\psi(\xi) = -ib \cdot \xi + \frac{1}{2} \xi \cdot Q \xi + \int_{y \neq 0} \left( 1-e^{iy \cdot \xi} + iy \cdot \xi \I_{(0,1)}(|y|) \right) \, \nu(dy), \quad \xi \in \mbb{R}^d,
\end{equation*}
where the \emph{L\'evy triplet} $(b,Q,\nu)$ consists of a vector $b \in \mbb{R}^d$ (\emph{drift vector}), a symmetric positive semi-definite matrix $Q \in \mbb{R}^{d \times d}$ (\emph{diffusion matrix}) and a measure $\nu$ on $\mbb{R}^d \backslash \{0\}$ with $\int_{y \neq 0} \min\{1,|y|^2\} \, \nu(dy)<\infty$  (\emph{L\'evy measure}). Our standard reference for L\'evy processes is the monograph \cite{sato} by Sato. By the independence and stationarity of the increments, every L\'evy process is a time-homogeneous Markov process, i.e.\ $P_t f(x) := \mbb{E}f(x+X_t)$ defines a Markov semigroup on $\mc{B}_b(\mbb{R}^d)$. We denote by $(A,\mc{D}(A))$ the \emph{(weak) infinitesimal generator}, \begin{align*}
	\mc{D}(A) &:= \left\{f \in \mc{B}_b(\mbb{R}^d); \exists g \in \mc{B}_b(\mbb{R}^d)\, \forall x \in \mbb{R}^d \::\: \lim_{t \to 0} \frac{\mbb{E}f(x+X_t)-f(x)}{t} = g(x) \right\}, \\
	Af(x) &:= \lim_{t \to 0} \frac{\mbb{E}f(x+X_t)-f(x)}{t}, \quad f \in \mc{D}(A).
\end{align*}
If $f \in C_b^2(\mbb{R}^d)$ then $f \in \mc{D}(A)$ and \begin{align}
	Af(x) = b \cdot \nabla f(x) + \frac{1}{2} \tr(Q \cdot \nabla^2 f(x)) + \int_{y \neq 0} \left( f(x+y)-f(x)-\nabla f(x) \cdot y \I_{(0,1)}(|y|) \right) \, \nu(dy). \label{gen}
\end{align}
Restricted to $C_c^{\infty}(\mbb{R}^d)$, the infinitesimal generator is a pseudo-differential operator with symbol $\psi$,  \begin{equation*}
	Af(x) = - \int_{\mbb{R}^d} \psi(\xi) e^{ix \cdot \xi} \hat{f}(\xi) \, d\xi, \qquad f \in C_c^{\infty}(\mbb{R}^d),
\end{equation*}
where $\hat{f}(\xi) = (2\pi)^{-d} \int_{\mbb{R}^d} f(x) e^{-ix \cdot \xi} \, dx$ is the Fourier transform of $f$. The pseudo-differential $A^*$ with symbol $\overline{\psi(\xi)}=\psi(-\xi)$ is the \emph{adjoint} of $A$, in the sense that, \begin{equation*}
	\forall f,g \in C_c^{\infty}(\mbb{R}^d) \::\: \int_{\mbb{R}^d} Af(x) g(x) \, dx = \int_{\mbb{R}^d} f(x) A^* g(x) \, dx.
\end{equation*}
Given an open set $U \subseteq \mbb{R}^d$ and $g \in L^{\infty}(U)$, a function $f$ is called a \emph{weak solution to $Af=g$ in $U$} if \begin{equation}
	\forall \varphi \in C_c^{\infty}(U) \::\: \int_{\mbb{R}^d} f(x) A^* \varphi(x) \, dx = \int_U g(x) \varphi(x) \, dx. \label{weak}
\end{equation}
It is implicitly assumed that the integral on the left-hand side exists; a sufficient condition is $f \in L^{\infty}(\mbb{R}^d)$, see e.g.\ \cite[Lemma 2.1]{fall16} and \cite[Proposition 2.1]{liouville} for milder growth conditions on $f$.

\section{Proofs} \label{p}

In this section we present the proofs of our main results. Corollary~\ref{main-3} and Corollary~\ref{main-5} are consequences of Theorem~\ref{main-1}, and therefore the main part is to establish Theorem~\ref{main-1}. The idea is to combine the global Schauder estimates from \cite{reg-levy} with a truncation technique to establish interior Schauder estimates. We start with the following auxiliary result.

\begin{lem} \label{p-1}
	Let $(X_t)_{t \geq 0}$ be a L\'evy process with generator $(A,\mc{D}(A))$ and L\'evy triplet $(b,Q,\nu)$. If $f,g \in C_b^2(\mbb{R}^d)$ then $f g \in \mc{D}(A)$ and \begin{equation*}
		A(f g) = g Af + f Ag + \Gamma(f,g)
	\end{equation*}
	where \begin{equation*}
		\Gamma(f,g)(x) := \nabla f(x) \cdot Q \nabla g(x) + \int_{y \neq 0} (f(x+y)-f(x))(g(x+y)-g(x)) \, \nu(dy), \qquad x \in \mbb{R}^d,
	\end{equation*}
	is the \textup{Carr\'e du Champ operator}. 
\end{lem}

\begin{proof}
	Clearly, $f \cdot g \in C_b^2(\mbb{R}^d) \subseteq \mc{D}(A)$. The identity for $A(f \cdot g)$ follows by applying \eqref{gen} for $f \cdot g$ and rearranging the terms.
\end{proof}

Let us mention that the regularity assumptions in Lemma~\ref{p-1} can be relaxed. Roughly speaking, the identity holds whenever $f,g \in \mc{D}(A)$ are sufficiently smooth to make sense of $\Gamma(f,g)$; e.g.\ if $Q=0$ then $f,g$ need to satisfy a certain H\"{o}lder condition, see \cite[Theorem 4.3]{reg-levy} and \cite{ihke}. \par \medskip

The following a priori estimate is the core of the proof of our first main result, Theorem~\ref{main-1}. 

\begin{prop} \label{p-3}
	Let $(X_t)_{t \geq 0}$ be a L\'evy process with generator $(A,\mc{D}(A))$, characteristic exponent $\psi$ and L\'evy triplet $(b,Q,\nu)$. If \eqref{A1}-\eqref{A3} hold, then there exists for every $R>0$ and $\delta>0$ some constant $c>0$ such that \begin{equation*}
		\|f\|_{\mc{C}_b^{\alpha}(B(x,R))} \leq c \left(\|f\|_{\infty} + \|Af\|_{\infty,B(x,R+\delta)}\right)
	\end{equation*}
	for all $f \in C_b^2(\mbb{R}^d)$ and $x \in \mbb{R}^d$; here $\alpha \in (0,2]$ denotes the constant from \eqref{A2}.
\end{prop}

In the proof of Proposition~\ref{p-3} we will use the elementary inequalities \begin{align}
	\sup_{|h| \leq R} \frac{|\Delta_h^3 f(x)|}{|h|^{\alpha}} 
	&\leq 8 r^{-\alpha} \|f\|_{\infty} + \sup_{|h| \leq r} \frac{|\Delta_h^3 f(x)|}{|h|^{\alpha}} \label{p-eq6} \\
	&\leq 8 r^{-\alpha} \|f\|_{\infty} + r^{\eps} \sup_{|h| \leq r} \frac{|\Delta_h^3 f(x)|}{|h|^{\alpha+\eps}}, \label{p-eq7}
\end{align}
which hold for any $0<r<R \leq 1$, $\eps >0$, $f \in \mc{B}_b(\mbb{R}^d)$ and $x \in \mbb{R}^d$.

\begin{proof}[Proof of Proposition~\ref{p-3}]
	For simplicity of notation we consider $x_0 =0$ and $R=\delta=1$, i.e.\ we need to show \begin{equation*}
		\|f\|_{\mc{C}_b^{\alpha}(B(0,1))} \leq c\left(\|f\|_{\infty} + \|Af\|_{\infty,B(0,2)}\right).
	\end{equation*}
	Let $\chi \in C_c^{\infty}(\mbb{R}^d)$ be such that $\I_{B(0,1/4)} \leq \chi \leq \I_{B(0,1/2)}$. For $x \in B(0,2)$ set $r_x := d(x,B(0,2)^c)$ and $\chi^{(x)}(y) := \chi(r_x^{-1}(y-x))$. As $r_x \leq 2$, it follows from the chain rule that \begin{equation*}
		\|\chi^{(x)}\|_{C_b^2(\mbb{R}^d)} \leq (1+r^{-2}) \|\chi\|_{C_b^2(\mbb{R}^d)} \leq 5r^{-2} \|\chi\|_{C_b^2(\mbb{R}^d)}.
	\end{equation*}
	We split the proofs in several steps. \par
	\textbf{Step 1:} There exist constants $\eps \in (0,\alpha)$, $\varrho>0$ and $C_1>0$ (not depending on $f$) such that  \begin{equation}
		\|f \chi^{(x)}\|_{\mc{C}_b^{\alpha}(\mbb{R}^d)} \leq C_1 \left(r_x^{-\alpha-\varrho} \|f\|_{\infty} + \|Af\|_{\infty,B(0,2)} +  \sup_{z \in B(x,7r_x /8)} K(x,z) \right) \label{p-eq8}
	\end{equation}
	for all $x \in B(0,2)$, where \begin{equation}
		K(x,z) := \sup_{|h| \leq r_x/32} \frac{|\Delta_h^3 f(z)|}{|h|^{\alpha-\eps}} \label{p-eq9}
	\end{equation}
	\emph{Indeed:} We fix $x \in (0,2)$ and write $r := r_x$ for brevity. Since $f \chi^{(x)}$ is twice continuously differentiable and vanishing at infinity, it is contained in the domain of the strong infinitesimal generator, and so it follows from \cite[Theorem 1.1]{reg-levy} that \begin{equation*}
		\|f \chi^{(x)} \|_{\mc{C}_b^{\alpha}(\mbb{R}^d)} \leq c_1 \left(\|f\|_{\infty}  + \|A(f \chi^{(x)})\|_{\infty}\right)
	\end{equation*}
	for some constant $c_1>0$ which does not depend on $f$ and $x$. Hence, by Lemma~\ref{p-1}, \begin{align*}
		\|f \chi^{(x)}\|_{\mc{C}_b^{\alpha}(\mbb{R}^d)} 
		\leq c_1 \left( \|f\|_{\infty} + \|\chi^{(x)} Af\|_{\infty} + \|f A \chi^{(x)} \|_{\infty}  + \|\Gamma(f,\chi^{(x)})\|_{\infty} \right).
	\end{align*}
	As $0 \leq \chi^{(x)} \leq 1$, $\spt \chi^{(x)} \subseteq B(0,2)$ and \begin{equation*}
		\|A \chi^{(x)}\|_{\infty}
		\leq c_2 \|\chi^{(x)}\|_{C_b^2(\mbb{R}^d)} 
		\leq 5c_2 r^{-2} \|\chi\|_{C_b^2(\mbb{R}^d)},
	\end{equation*}
	this gives \begin{equation}
		\|f \chi^{(x)}\|_{\mc{C}_b^{\alpha}(\mbb{R}^d)}
		\leq (1+5c_2) \left( r^{-2} \|f\|_{\infty} + \|Af\|_{\infty,B(0,2)} + \|\Gamma(f,\chi^{(x)})\|_{\infty} \right). \label{p-eq10}
	\end{equation}
	It remains to estimate the supremum norm of the term involving the Carr\'e du champ operator. For $z \in \mbb{R}^d \backslash B(x,3r/4)$ we have $\chi^{(x)}=0$ on $B(z,r/4)$, and therefore \begin{equation*}
		\Gamma(f,\chi^{(x)})(z) = \int_{|y|>r/4} \left( f(z+y)-f(z) \right) \left( \chi^{(x)}(z+y)-\chi^{(x)}(z) \right) \, \nu(dy)
	\end{equation*}
	implying \begin{align}
		|\Gamma(f,\chi^{(x)})(z)| 
		\leq 4 \|f\|_{\infty} \int_{|y|>r/4} \, \nu(dy)
		\leq 4^{1+\gamma} r^{-\gamma}  \|f\|_{\infty} \int_{y \neq 0} \min\{1,|y|^{\gamma}\} \, \nu(dy) \label{p-eq12}.
	\end{align}
	Since $\int_{y \neq 0} \min\{1,|y|^{\gamma}\} \, \nu(dy)<\infty$, cf.\ \eqref{A3}, and $r \leq 2$, this implies \begin{equation*}
		|\Gamma(f,\chi^{(x)})(z)|
		\leq c_3 r^{-\alpha-\varrho} \|f\|_{\infty}, \qquad z \in \mbb{R}^d \backslash B(x,3r/4),
	\end{equation*}
	for suitable constants $c_3>0$ and $\varrho>0$, which is an estimate of the desired form. For $z \in B(x,3r/4)$ we consider the local term and the non-local term separately, i.e.\ we write $\Gamma(f,\chi^{(x)})(z) = D(z) + I(z)$ where \begin{equation*}
		D(z) := \nabla f(z) \cdot Q \nabla \chi^{(x)}(z) \quad \text{and} \quad I(z) := \int_{y \neq 0} (f(z+y)-f(z))(\chi^{(x)}(z+y)-\chi^{(x)}(z)) \, \nu(dy).
	\end{equation*}
	For the local term it clearly suffices to consider the case $Q \neq 0$. If $Q \neq 0$ then, by \eqref{A3}, $\alpha>1$ and therefore we can choose $\eps \in (0,1)$ such that $\alpha-2\eps>1$. Clearly, \begin{align*}
		|D(z)| 
		\leq \|\chi^{(x)}\|_{C_b^1(\mbb{R}^d)} \|Q\| |\nabla f(z)|
		\leq 3r^{-1} \|Q\| \|\chi\|_{C_b^1(\mbb{R}^d)} |\nabla f(z)|.
	\end{align*}
	Since there exists a constant $c_4>0$ such that \begin{equation*}
		|\nabla f(z)|
		\leq c_4 r^{-\alpha} \|f\|_{\infty} + c_4 \sup_{|y-z| \leq r/16} \sup_{|h| \leq 1} \frac{|\Delta_h^3 f(y)|}{|h|^{\alpha-2\eps}},
	\end{equation*}
	cf. Lemma~\ref{def-11}\eqref{def-11-ii}, it follows from \eqref{p-eq6} and \eqref{p-eq7} that \begin{align*}
		|\nabla f(z)|
		&\leq c_4' r^{-2\theta-\alpha} \|f\|_{\infty} + c_4' \sup \left\{ \frac{|\Delta_h^3 f(y)|}{|h|^{\alpha-2\eps}}; \, y \in B(x,7r/8),\, |h| \leq  \frac{1}{32} \left(\frac{r}{2}\right)^{\theta} \right\} \\
		&\leq c_4' r^{-2\theta-\alpha} \|f\|_{\infty} + c_4'' r^{\theta \eps} \sup \left\{ \frac{|\Delta_h^3 f(y)|}{|h|^{\alpha-\eps}}; \, y \in B(x,7r/8),\, |h| \leq  \frac{1}{32} \left(\frac{r}{2}\right)^{\theta} \right\}
	\end{align*}
	for any $\theta>0$ and $z \in B(x,3r/4)$. Choosing $\theta := 1/\eps$ we get \begin{equation*}
		|D(z)| \leq c_5 r^{-1-2/\eps-\alpha} \|f\|_{\infty} + c_5 \sup \left\{ \frac{|\Delta_h^3 f(y)|}{|h|^{\alpha-\eps}}; \, y \in B(x,7r/8),\, |h| \leq  \frac{1}{32} \left(\frac{r}{2}\right)^{\theta} \right\}
	\end{equation*} 
	for all $z \in B(x,3r/4)$. As $r/2 \leq 1$ and $\theta>1$, the supremum over $|h| \leq \frac{1}{32} (r/2)^{\theta}$ is less or equal than the supremum over $|h| \leq \frac{1}{32} (r/2)$, and so \begin{equation*}
		|D(z)| \leq  c_5 r^{-1-2/\eps-\alpha} \|f\|_{\infty} + c_5 \sup_{y \in B(x,7r/8)} K(x,y)
	\end{equation*}
	with $K$ defined in \eqref{p-eq9}. 	It remains to estimate the non-local term.  We consider the cases $\gamma=2$ and $\gamma \in (0,2)$ separately. If $\gamma=2$, then by \eqref{A3} $\alpha>1$. Choose $\eps \in (0,1)$ such that $\alpha-2\eps>1$ and set $\theta=1/\eps$. By the Lipschitz continuity of $\chi^{(x)}$, \begin{align*}
		|I(z)|
		&\leq \|\chi^{(x)}\|_{C_b^1(\mbb{R}^d)} \int_{|y| \leq r/32} |f(y+z)-f(z)| \, |y| \, \nu(dy) + 4 \|f\|_{\infty} \int_{|y|>r/32} \, \nu(dy).
	\end{align*}
	Since $\int_{y \neq 0} \min\{1,|y|^{\gamma}\} \, \nu(dy) < \infty$ and $\|\chi^{(x)}\|_{C_b^1(\mbb{R}^d)} \leq 3r^{-1} \|\chi\|_{C_b^1(\mbb{R}^d)}$, this implies that there exists a finite constant $c_6>0$ such that 
	\begin{align*}
		|I(z)| 
		& \leq c_6 r^{-1} \|\chi\|_{C_b^1(\mbb{R}^d)} \sup_{|h| \leq r/32} |\nabla f(z+h)| \int_{|y| \leq r/32} |y|^2 \, \nu(dy)  +c_6 r^{-\gamma}  \|f\|_{\infty}. 
	\end{align*}
	 For the first term on the right-hand side we can now use a reasoning similar to that for the local term and get the required estimate. Finally we consider the case $\gamma<2$.  Fix $\eps \in (0,1)$ such that $\max\{0,\gamma-1\} < \alpha-2 \eps$ -- it exists because of \eqref{A3} -- and set $\theta := 1/\eps$. We have \begin{align*}
		|I(z)|
		&\leq \|\chi^{(x)}\|_{C_b^1(\mbb{R}^d)} \sup_{|h| \leq 1} \frac{|f(z+h)-f(z)|}{|h|^{\max\{0,\gamma-1\}}} \int_{|y| \leq 1} |y|^{\gamma} \, \nu(dy) + 4 \|f\|_{\infty} \int_{|y|>1} \, \nu(dy).
	\end{align*}
	Using that $\int \min\{1,|y|^{\gamma}\} \, \nu(dy)< \infty$ and $\|\chi^{(x)}\|_{C_b^1(\mbb{R}^d)} \leq 3r^{-1} \|\chi\|_{C_b^1(\mbb{R}^d)}$, we find a finite constant $c_7>0$ such that \begin{equation*}
		|I(z)|
		\leq c_7 \|f\|_{\infty} + c_7 r^{-1} \sup_{|h| \leq 1} \frac{|f(z+h)-f(z)|}{|h|^{\max\{0,\gamma-1\}}}.
	\end{equation*}
	Since $\max\{0,\gamma-1\}<1$, an application of Lemma~\ref{def-11}\eqref{def-11-i} shows that \begin{align*}
		\sup_{|h| \leq 1} \frac{|f(z+h)-f(z)|}{|h|^{\max\{0,\gamma-1\}}}
		&\leq c_8 r^{- 1 } \|f\|_{\infty} + c_8 \sup_{|h| \leq 1} \sup_{|y-z| \leq r/16} \frac{|\Delta_h^3 f(y)|}{|h|^{\max\{0,\gamma-1\}}}.
	\end{align*}
	Hence, by \eqref{p-eq6} and \eqref{p-eq7},
	\begin{align*}
		\sup_{|y| \leq 1} \frac{|f(z+y)-f(z)|}{|y|^{\max\{0,\gamma-1\}}}
		&\leq c_8' r^{-\theta} \|f\|_{\infty} + c_8 r^{\theta \eps} \sup \left\{ \frac{|\Delta_h^3 f(y)|}{|h|^{\alpha-\eps}}; \, y \in B(x,7r/8),\, |h| \leq  \frac{1}{32} \left(\frac{r}{2}\right)^{\theta} \right\}
	\end{align*}
	 for all $z \in B(x,3r/4)$. Recalling that $\theta \eps=1$ and $r \leq 2$ we conclude that \begin{align*}
		|I(z)|
		&\leq c_9 r^{- \theta-1} \|f\|_{\infty} + c_9 \sup \left\{ \frac{|\Delta_h^3 f(y)|}{|h|^{\alpha-\eps}}; \, y \in B(x,7r/8),\, |h| \leq  \frac{1}{32} \left(\frac{r}{2}\right)^{\theta} \right\} \\
		&\leq  c_9 r^{- \theta-1} \|f\|_{\infty} + c_9  \sup_{y \in B(x,7r/8)} K(x,y).
	\end{align*}
	\textbf{Step 2:} There exists a constant $C_2>0$ such that  \begin{equation}
		|f|_{B(0,2),\alpha,\varrho} \leq C_2 \left( \|f\|_{\infty} + \|Af\|_{\infty,B(0,2)} + |f|_{B(0,2),\alpha-\eps,\varrho} \right), \label{p-eq15}
	\end{equation}
	where $\varrho$, $\eps$ are the constants from Step 1 and \begin{equation}
		|f|_{B(0,2),\sigma,\kappa} := \sup_{x \in B(0,2)} \sup_{|h| \leq r_x/4} r_x^{\kappa+\sigma} \frac{|\Delta_h^3 f(x)|}{|h|^{\sigma}}, \qquad \kappa>0, \sigma \in (0,3). \label{p-eq16}
	\end{equation}
	\emph{Indeed:} For $z \in B(x,7r_x/8)$ and $x \in B(0,2)$ we have $r_z \geq r_x/8$, and therefore the mapping $K$ defined in \eqref{p-eq9} satisfies \begin{align*}
		r_x^{\alpha+\varrho} \sup_{z \in B(x,7r_x/8)} K(x,z)
		&\leq 8^{\alpha+\varrho} \sup_{z \in B(x,7r_x/8)} \sup_{|h| \leq r_x/32} r_z^{\alpha+\varrho}   \frac{|\Delta_h^3 f(z)|}{|h|^{\alpha-\eps}} \\
		&\leq 8^{\alpha+\varrho} \sup_{z \in B(0,2)} \sup_{|h| \leq r_z/4} r_z^{\alpha+\varrho} \frac{|\Delta_h^3 f(z)|}{|h|^{\alpha-\eps}} \\
		&\leq 2 \, 8^{\alpha+\varrho} |f|_{B(0,2),\alpha-\eps,\varrho}.
	\end{align*}
	Consequently, it follows from Step 1 that \begin{equation*}
		r_x^{\varrho+\alpha} \|f \chi^{(x)} \|_{\mc{C}_b^{\alpha}(\mbb{R}^d)} 
		\leq C_1 \left( \|f\|_{\infty} + \|Af\|_{\infty,B(0,2)} + 2 \, 8^{\alpha+\varrho} |f|_{B(0,2),\alpha-\eps,\varrho} \right) 
	\end{equation*}
	for all $x \in B(0,2)$. As $\alpha \in (0,2]$, this gives \begin{align}
		r_x^{\varrho+\alpha} \sup_{|h| \leq 1} \sup_{z \in \mbb{R}^d} \frac{|\Delta_h^3 (f \chi^{(x)})(z)|}{|h|^{\alpha}}
		\leq c_2 \left( \|f\|_{\infty} + \|Af\|_{\infty,B(0,2)} + |f|_{B(0,2),\alpha-\eps,\varrho} \right)  \label{p-eq20}
	\end{align}
	for all $x \in B(0,2)$ and some constant $c_2>0$ not depending on $x$ and $f$, cf.\ \eqref{def-eq5}. Since $\chi^{(x)}=1$ on $B(x,r_x/4)$, this implies \begin{equation*}
		r_x^{\varrho+\alpha} \sup_{|h| \leq r_x/12} \frac{|\Delta_h^3 f(x)|}{|h|^\alpha} \leq c_2 \left( \|f\|_{\infty} + \|Af\|_{\infty,B(0,2)} + |f|_{B(0,2),\alpha-\eps,\varrho} \right).
	\end{equation*}
	On the other hand, we have \begin{equation*}
		r_x^{\varrho+\alpha} \sup_{r_x/12 \leq |h| \leq r_x/4} \frac{|\Delta_h^3 f(x)|}{|h|^\alpha}
		\leq 8 r_x^{\varrho+\alpha} \|f\|_{\infty} \left( \frac{r_x}{12} \right)^{-\alpha} 
		\leq c_3 \|f\|_{\infty}
	\end{equation*}
	for some uniform constant $c_3>0$. Combining both estimates yields \begin{equation*}
		r_x^{\varrho+\alpha} \sup_{|h| \leq r_x/4} \frac{|\Delta_h^3 f(x)|}{|h|^\alpha} \leq c_4 \left( \|f\|_{\infty} + \|Af\|_{\infty,B(0,2)} + |f|_{B(0,2),\alpha-\eps,\varrho} \right)
	\end{equation*}
	for all $x \in B(0,2)$, and this proves \eqref{p-eq15}. \par
	\textbf{Conclusion of the proof}: Choose $\delta \in (0,\frac{1}{4})$  sufficiently small such that $C_2 \delta \leq \frac{1}{2}$  for the constant $C_2$ from Step 2. By definition, cf.\ \eqref{p-eq16}, \begin{equation*}
		|f|_{B(0,2),\alpha-\eps,\varrho}
		= \sup_{x \in B(0,2)} \sup_{h: |h/r_x| \leq 1/4} r_x^{\varrho} \left( \frac{r_x}{|h|} \right)^{\alpha-\eps} |\Delta_h^3 f(x)|.
	\end{equation*}
	Since \begin{equation*}
		\sup_{h: |h/r_x| \leq \delta^{1/\eps}} r_x^{\varrho} \left( \frac{r_x}{|h|} \right)^{\alpha-\eps} |\Delta_h^3 f(x)|
		\leq \delta \sup_{h: |h/r_x| \leq \delta^{1/\eps}} r_x^{\varrho} \left( \frac{r_x}{|h|} \right)^{\alpha} |\Delta_h^3 f(x)|
		\leq \delta |f|_{B(0,2),\alpha,\varrho}
	\end{equation*}
	and \begin{equation*}
		\sup_{h: \delta^{1/\eps}< |h/r_x| \leq 1/4} r_x^{\varrho} \left( \frac{r_x}{|h|} \right)^{\alpha-\eps} |\Delta_h^3 f(x)|
		\leq \delta^{-\alpha/\eps} \sup_{h: \delta^{1/\eps}< |h/r_x| \leq 1/4} r_x^{\varrho} |\Delta_h^3 f(x)| 
		\leq 8 \delta^{-\alpha/\eps} 2^{\varrho} \|f\|_{\infty},
	\end{equation*}
	it follows that there exists a constant $c_1>0$ such that \begin{equation}
		|f|_{B(0,2),\alpha-\eps,\varrho} \leq \delta |f|_{B(0,2),\alpha,\varrho} + c_1 \|f\|_{\infty}. \label{p-eq22}
	\end{equation}
	As $C_2 \delta \leq \frac{1}{2}$, we find from Step 2 and \eqref{p-eq22} that \begin{align*}
		|f|_{B(0,2),\alpha,\varrho}
		&\leq C_2 \left(\|f\|_{\infty} + \|Af\|_{\infty,B(0,2)} + |f|_{B(0,2),\alpha-\eps,\varrho}\right) \\
		&\leq C_2 \left((1+c_1) \|f\|_{\infty} + \|Af\|_{\infty,B(0,2)} \right) + \frac{1}{2} |f|_{B(0,2),\alpha,\varrho},
	\end{align*}
	i.e. \begin{equation*}
		|f|_{B(0,2),\alpha,\varrho}
		\leq c_2 \left(\|f\|_{\infty} + \|Af\|_{\infty,B(0,2)}\right).
	\end{equation*}
	On the other hand, we also have \begin{align*}
		|f|_{B(0,2),\alpha,\varrho} \geq \sup_{x \in B(0,1)} \sup_{|h| \leq 1/4} \frac{|\Delta_h^3 f(x)|}{|h|^{\alpha}}.
	\end{align*}
	Combining both estimates and applying Lemma~\ref{def-11} proves the assertion.
\end{proof}

The seminorms $|f|_{U,\alpha,\varrho}$ which we introduced in \eqref{p-eq15} are closely related to seminorms which appear in the study of Schauder estimates for second order differential operators, cf.\ \cite{gilbarg}; our definition is inspired by H\"{o}lder--Zygmund norms whereas the seminorms in \cite{gilbarg} are based on   ``classical'' H\"{o}lder norms. \par \medskip

In order to apply the a priori estimate from Proposition~\ref{p-3}, we have to approximate the weak solution $f$ by a sequence $(f_k)_{k \geq 1}$ of twice differentiable functions; it is a natural idea to consider $f_k := f \ast \chi_k$ for a suitable sequence of mollifiers $(\chi_k)_{k \geq 1}$. To make this approximation work, we need to know that $Af_k$ is on $U$ close to $Af=g$, and this is what the next lemma is about.

\begin{lem} \label{p-2}
	Let $(X_t)_{t \geq 0}$ be a L\'evy process with infinitesimal generator $(A,\mc{D}(A))$. Let $f \in L^{\infty}(\mbb{R}^d)$ be a weak solution to the equation $Af = g$ in $U$ for  $g \in L^{\infty}(U)$ and $U \subseteq \mbb{R}^d$ open. If $\chi \in C_c^{\infty}(B(0,r))$ for some r$>0$, then $f \ast \chi \in \mc{D}(A)$ and $A(f \ast \chi)(x) = (g \ast \chi)(x)$ for all $x \in U$ with $B(x,r) \subseteq U$.
\end{lem}

\begin{proof}
	It is well known that $u:=f \ast \chi \in C_b^{\infty}(\mbb{R}^d)$ and \begin{equation}
	\partial^{\alpha} u = f \ast (\partial^{\alpha} \chi), \qquad \alpha \in \mbb{N}_0^d, \label{p-eq3}
	\end{equation}
	see e.g.\ \cite{mims}. In particular, $u \in C_b^2(\mbb{R}^d) \subseteq \mc{D}(A)$ and \begin{equation*}
		Au(x) = b \cdot \nabla u(x) + \frac{1}{2} \tr(Q \cdot \nabla^2 u(x)) + \int_{y \neq 0} (u(x+y)-u(x) - \nabla u(x) \cdot y \I_{(0,1)}(|y|)) \, \nu(dy).
	\end{equation*}
	Using \eqref{p-eq3} and Fubini's theorem it follows that \begin{equation*}
		Au(x) = (f \ast A\chi)(x) = \int_{\mbb{R}^d} f(y) (A \chi)(x-y) \, dy
	\end{equation*} 
	for all $x \in \mbb{R}^d$. As \begin{equation*}
		(A \chi)(x-y) = (A^* \chi(x-\cdot))(y)
	\end{equation*}
	for the adjoint $A^*$, we get \begin{equation*}
		Au(x) = \int_{\mbb{R}^d} f(y) (A^* \chi(x-\cdot))(y) \, dy.
	\end{equation*}
	If $x \in U$ is such that $B(x,r) \subseteq U$, then $\spt \chi(x-\cdot) \subseteq U$ and so it follows from the definition of the weak solution, cf.\ \eqref{weak}, that \begin{equation*}
		Au(x)= \int_{\mbb{R}^d} g(y) \chi(x-y) \, dy= (g \ast \chi)(x)
	\end{equation*}
	for any such $x$.
\end{proof}

We are now ready to prove Theorem~\ref{main-1}.

\begin{proof}[Proof of Theorem~\ref{main-1}] 
Let $f \in L^{\infty}(\mbb{R}^d)$ and $g \in L^{\infty}(U)$ be such that $Af = g$ weakly in $U$. Pick $\chi \in C_c^{\infty}(\mbb{R}^d)$ such that $\chi \geq 0$, $\spt \chi \subseteq B(0,1)$, $\int \chi(x) \, dx = 1$ and set $\chi_k(x) := k^d \chi(kx)$. If we define \begin{equation*}
	f_k(x) := (f \ast \chi_k)(x) := \int_{\mbb{R}^d} f(y) \chi_k(x-y) \, dy,
\end{equation*}
then Lemma~\ref{p-2} shows that $f_k \in C_b^2(\mbb{R}^d) \subseteq \mc{D}(A)$ and \begin{equation}
	Af_k(x) = g_k(x) := (g \ast \chi_k)(x) \quad \text{for all $x \in U$ with $B(x,1/k) \subseteq U$}. \label{p-eq41}
\end{equation}
\begin{enumerate}[wide, labelwidth=!, labelindent=0pt] 
	\item Set $U_{\delta} := \{x \in U; d(x,U^c)>\delta\}$ for $\delta>0$. By Proposition~\ref{p-3}, there exists a constant $c_1 = c_1(\delta)$ such that \begin{equation*}
	\|f_k\|_{\mc{C}_b^{\alpha}(B(x,\delta/2))} \leq c_1 \left(\|f_k\|_{\infty} + \|Af_k\|_{\infty,U_{\delta/4}}\right)
	\end{equation*}
	for all $x \in U_{\delta}$ and $k \geq 1$. Choosing $k \gg 1$ sufficiently large such that $\frac{1}{k} < \frac{\delta}{4}$ we find from \eqref{p-eq41} and $\|f_k\|_{\infty} \leq \|f\|_{L^{\infty}(\mbb{R}^d)}$ that \begin{equation}
		\|f_k\|_{\mc{C}_b^{\alpha}(B(x,\delta/2))} \leq c_1 \left(\|f\|_{L^{\infty}(\mbb{R}^d)} + \|g\|_{L^{\infty}(U)}\right) \fa x \in U_{\delta}. \label{p-eq43}
	\end{equation}
	It follows from the Arzel\`a-Ascoli theorem that there exists a continuous function $f^{(\delta)}: U_{\delta} \to \mbb{R}$ such that a subsequence $f_{k_j}$ converges pointwise to $f^{(\delta)}$ on $U_{\delta}$. By \eqref{p-eq43}, we have \begin{equation*}
		\frac{|\Delta_h^N f^{(\delta)}(x)|}{|h|^{\alpha}}
		= \lim_{j \to \infty} \frac{|\Delta_h^N f_{k_j}(x)|}{|h|^{\alpha}}
		\leq c_1  \left(\|f\|_{L^{\infty}(\mbb{R}^d)} + \|g\|_{L^{\infty}(U)}\right)
	\end{equation*}
	for all $x \in U_{2\delta}$ and $|h| \leq \delta/3$ where $N \in \{1,2,3\}$ is the smallest integer larger than $\alpha$. Hence, \begin{equation*}
		\|f^{(\delta)}\|_{\mc{C}_b^{\alpha}(U_{2\delta})} 
		\leq c_2 \left(\|f\|_{L^{\infty}(\mbb{R}^d)} + \|g\|_{L^{\infty}(U)}\right).
	\end{equation*}
	On the other hand, $f_k \to f$ in $L^1(dx)$ and so $f=f^{(\delta)}$ Lebesgue almost everywhere on $U_{\delta}$. As $f^{(\delta')}= f^{(\delta)}$ on $U_{\delta}$ for $\delta' < \delta$, the mapping \begin{equation*}
		\tilde{f}(x) := \begin{cases} f^{(\delta)}(x), & \text{if $x \in U_{\delta}$, $\delta>0$}, \\ f(x), & \text{if $x \in U^c$}, \end{cases}
	\end{equation*}
	is well defined. Clearly, $\tilde{f}=f$ Lebesgue almost everywhere and the interior Schauder estimate \begin{equation*}
		\|\tilde{f}\|_{\mc{C}_b^{\alpha}(U_{\delta})} 
		\leq C_{\delta}  \left(\|f\|_{L^{\infty}(\mbb{R}^d)} + \|g\|_{L^{\infty}(U)}\right)
	\end{equation*}
	holds for all $\delta>0$. 
	\item Assume now additionally that $f \in \mc{C}_b^{\kappa}(\mbb{R}^d)$ and $g \in \mc{C}_b^{\kappa}(U)$ for some $\kappa>0$. Since $A|_{C_b^2(\mbb{R}^d)}$ commutes with the shift operator, i.e. \begin{equation*}
		(A\varphi)(x+x_0) = (A\varphi(\cdot+x_0))(x), \qquad \varphi \in C_b^2(\mbb{R}^d), \,x,x_0 \in \mbb{R}^d,
	\end{equation*}
	cf.\ \eqref{gen}, it follows easily by induction that \begin{equation*}
		A(\Delta_h^N \varphi) = \Delta_h^N (A \varphi) \fa \varphi \in C_b^2(\mbb{R}^d),\, N \in \mbb{N}.
	\end{equation*}
	Hence, by \eqref{p-eq41}, \begin{equation}
		A(\Delta_h^N f_k)(x) = (\Delta_h^N Af_k)(x) = \Delta_h^N g_k(x) \label{p-eq47}
	\end{equation}
	for all $N \in \mathbb{N}$, $x \in U_{\delta}$, $|h| < \delta/(2N)$ and $k \gg 1$ with $\frac{1}{k}<\frac{\delta}{2}$. Let $N \in \mbb{N}$ be the smallest number which is strictly larger than $\kappa$. Applying part~\eqref{main-1-i} to $x \mapsto \Delta_h^N f_k(x)$, we obtain from \eqref{p-eq47} that \begin{equation*}
		\|\Delta_h^N f_k\|_{\mc{C}_b^{\alpha}(U_{2\delta})} 
		\leq c_1 \left( \|\Delta_h^N f_k\|_{\infty} + \|\Delta_h^N g_k\|_{\infty, U_{\delta}}\right)
	\end{equation*}
	for every $\delta \in (0,1)$ and some constant $c_1=c_1(\delta)>0$ not depending on $f$, $g$ and $k$. Since $f \in C_b^{\kappa}(\mbb{R}^d)$ and $g \in \mc{C}_b^{\kappa}(U)$, this implies \begin{equation*}
		\|\Delta_h^N f_k\|_{\mc{C}_b^{\alpha}(U_{2\delta})} 
		\leq c_2 |h|^{\kappa} \left( \|f\|_{\mc{C}_b^{\kappa}(\mbb{R}^d)} + \|g\|_{\mc{C}_b^{\kappa}(U)} \right)
	\end{equation*}
	for all $|h| < \delta/(2N)$. If we denote by $M \in \mbb{N}$ the smallest number strictly larger than $\alpha$, then we get \begin{equation*}
		|\Delta_h^M \Delta_h^N f_k(x)| \leq c_3 |h|^{\kappa+\alpha}  \left( \|f\|_{\mc{C}_b^{\kappa}(\mbb{R}^d)} + \|g\|_{\mc{C}_b^{\kappa}(U)} \right)
	\end{equation*}
	for all $x \in U_{3\delta}$ and small $h$. The left-hand side converges to $\Delta_h^M \Delta_h^N f(x) = \Delta_h^{M+N} f(x)$ as $k \to \infty$ because $f$ is continuous. Applying Lemma~\ref{def-11} finishes the proof. \qedhere
\end{enumerate}
\end{proof}

For the proof of Corollary~\ref{main-3} we need another auxiliary result.

\begin{lem} \label{p-5}
	Let $(X_t)_{t \geq 0}$ be a L\'evy process with characteristic exponent $\psi$ satisfying the Hartman--Wintner condition \eqref{A1}. Let $f \in \mc{B}_b(\mbb{R}^d)$ be such that \begin{equation}
		\lim_{t \to 0} |\mbb{E}f(x+X_t)-f(x)| = 0, \qquad x \in U, \label{p-eq51}
	\end{equation}
	for an open set $U \subseteq \mbb{R}^d$. If $\tilde{f} \in C_b(U)$ is such that $f = \tilde{f}$ Lebesgue almost everywhere on $U$, then $f=\tilde{f}$ on $U$; in particular, $f|_U$ is continuous.
\end{lem}

\begin{proof}
	The function \begin{equation*}
		u := \tilde{f} \I_{U} + f \I_{U^c},
	\end{equation*}
	satisfies $u \in \mc{B}_b(\mbb{R}^d)$ and $f = u$ Lebesgue almost everywhere on $\mbb{R}^d$. Since the characteristic exponent $\psi$ satisfies the Hartman--Wintner condition, the law of $X_t$ has a density $p_t$ with respect to Lebesgue measure for $t>0$, and so \begin{equation*}
		\mbb{E} u(x+X_t) = \mbb{E} f(x+X_t) \fa x \in \mbb{R}^d,t>0. 
	\end{equation*}
	If we can show that \begin{equation}
		\lim_{t \to 0} \mbb{E} u(x+x_t) = u(x)= \tilde{f}(x) \fa x \in U, \label{p-eq53}
	\end{equation}
	then it follows immediately from \eqref{p-eq51} that \begin{equation*}
		f(x) =  \lim_{t \to 0} \mbb{E}f(x+X_t) = \lim_{t \to 0} \mbb{E}u(x+X_t) = \tilde{f}(x) \fa x \in U,
	\end{equation*}
	which proves the assertion.  To prove \eqref{p-eq53}, fix $x \in U$. As \begin{equation*}
		|\mbb{E}u(x+X_t)-u(x)|
		\leq \sup_{|y| \leq \delta} |u(x+y)-u(x)| + 2 \|u\|_{L^{\infty}(\mbb{R}^d)} \mbb{P} \left( \sup_{s \leq t} |X_s|>\delta \right),
	\end{equation*}
	we find from the right-continuity of the sample paths and the monotone convergence theorem that \begin{equation*}
		\limsup_{t \to 0} |\mbb{E}u(x+X_t)-u(x)| \leq \sup_{|y| \leq \delta} |u(x+y)-u(x)|.
	\end{equation*}
	Since $u$ is continuous at $x$, the right-hand side tends to $0$ as $\delta \to 0$, and this gives \eqref{p-eq53}.
\end{proof}
	
\begin{bem} \label{p-7}
	The proof of Lemma~\ref{p-5} shows the following statement: If $(X_t)_{t \geq 0}$ is a L\'evy process, then \begin{equation*}
		\lim_{t \to 0} \mbb{E}f(x+X_t) = f(x)
	\end{equation*}
	holds for any continuity point $x$ of $f \in \mc{B}_b(\mbb{R}^d)$; this is a localized version of the continuity of the semigroup $T_t f(x) := \mbb{E}f(x+X_t)$ at $t=0$.
\end{bem}

\begin{proof}[Proof of Corollary~\ref{main-3}]
	Let $\varphi \in C_c^{\infty}(U)$. Because of the uniform boundedness assumption \eqref{main-eq23}, it follows from the dominated convergence theorem that \begin{equation*}
		\int_{\mbb{R}^d} g(x) \varphi(x) \, dx
		= \lim_{t \to 0} \frac{1}{t}  \left( \mbb{E} \int_{\mbb{R}^d} f(x+X_t) \varphi(x) \, dx - \int_{\mbb{R}^d} f(x) \varphi(x) \, dx \right).
	\end{equation*}
	Performing a change of variables we get \begin{align*}
		\int_{\mbb{R}^d} g(x) \varphi(x) \, dx	
	 	&=\lim_{t \to 0} \frac{1}{t}  \left( \mbb{E} \int_{\mbb{R}^d} f(x) \varphi(x-X_t) \, dx - \int_{\mbb{R}^d} f(x) \varphi(x) \, dx \right) \\
	 	&= \int_{\mbb{R}^d} f(x) A^* \varphi(x) \, dx,
	\end{align*}
	i.e.\ $f$ is a weak solution to $Af=g$ on $U$. If $g \in L^{\infty}(U)$ then it follows from Theorem~\ref{main-1}\eqref{main-1-i} that there exists a function $\tilde{f} \in L^{\infty}(\mbb{R}^d)$ such that $f=\tilde{f}$ Lebesgue almost everywhere, $\tilde{f} \in C(U)$ and \begin{equation}
		\|\tilde{f}\|_{\mc{C}_b^{\alpha}(U_{\delta})} \leq C_{\delta} (\|f\|_{L^{\infty}(\mbb{R}^d)}+ \|g\|_{\infty,U}), \qquad \delta>0. \label{p-eq55}
	\end{equation}
	By Lemma~\ref{p-5}, we have $f=\tilde{f}$ on $U$, and consequently \eqref{p-eq55} holds with $\tilde{f}$ replaced by $f$; this proves the first assertion. The second assertion follows directly from Theorem~\ref{main-1}\eqref{main-1-ii}.
\end{proof}

\begin{proof}[Proof of Corollary~\ref{main-5}]

The first assertion is immediate from Corollary~\ref{main-3}\eqref{main-3-i}. For \eqref{main-5-ii}, we note that $Af \in \mc{C}_b^{\kappa}(\mbb{R}^d)$ implies $f \in \mc{C}_b^{\kappa+\alpha}(\mbb{R}^d)$ and \begin{equation*}
	\|f\|_{\mc{C}_b^{\kappa}(\mbb{R}^d)} 
	\leq \|f\|_{\mc{C}_b^{\kappa+\alpha}(\mbb{R}^d)} 
	\leq c \left( \|f\|_{\infty} + \|Af\|_{\infty} \right),
\end{equation*}
cf.\ \cite[Theorem 1.1]{reg-levy}, and so Corollary~\ref{main-3}\eqref{main-3-ii} applies.\end{proof}

\section{Examples} \label{ex}

In this section, we present examples of L\'evy processes for which interior Schauder estimates can be obtained from Theorem~\ref{main-1}, Corollary~\ref{main-3} and Corollary~\ref{main-5}. We start with two tools which are useful to construct wide classes of L\'evy processes satisfying the assumptions \eqref{A1}-\eqref{A3} of our main results.

\begin{prop} \label{ex-1}
	Let $(X_t^{(i)})_{t \geq 0}$, $i=1,2$, be independent $\mbb{R}^{d_i}$-valued L\'evy processes. \begin{enumerate}
	\item\label{ex-1-i} ($d_1=d_2$) If $(X_t^{(1)})_{t \geq 0}$ satisfies \eqref{A1} and \eqref{A2} for some $\alpha>0$, then the L\'evy process \begin{equation*}
		Y_t := X_t^{(1)}+X_t^{(2)}, \qquad t \geq 0,
	\end{equation*}
	satisfies \eqref{A1} and \eqref{A2} for the same constant $\alpha$.
	\item\label{ex-1-ii} If $(X_t^{(i)})_{t \geq 0}$ satisfies \eqref{A1} and \eqref{A2} for a constant $\alpha_i>0$, $i=1,2$, then \begin{equation*}
		Z_t := \begin{pmatrix} X_t^{(1)} \\ X_t^{(2)} \end{pmatrix}, \qquad t \geq 0,
	\end{equation*}
	satisfies \eqref{A1} and \eqref{A2} with $\alpha := \min\{\alpha_1,\alpha_2\}$.
\end{enumerate} \end{prop}

\begin{proof} \begin{enumerate}[wide, labelwidth=!, labelindent=0pt] 
	\item Set $d:=d_1=d_2$. The characteristic exponent of $(Y_t)_{t \geq 0}$ equals $\psi := \psi^{(1)} + \psi^{(2)}$ where $\psi^{(i)}$ is the characteristic exponent of $X_t^{(i)}$, $i=1,2$. In particular, $\re \psi \geq \re \psi^{(1)}$, and therefore the Hartman--Wintner condition \eqref{A1} for $\psi^{(1)}$ implies that $\psi$ satisfies \eqref{A1}. Consequently, the law of $Y_t$ has a density $p_t$ with respect to Lebesgue measure for $t>0$, and it satisfies \begin{equation*}
		p_t(x) = \int_{\mbb{R}^{d}} p_t^{(1)}(x-y) \, \mu_t(dy), \qquad x \in \mbb{R}^d,\,t>0,
	\end{equation*}
	where $p_t^{(1)}$ is the density of $X_t^{(1)}$ and $\mu_t$ is the law of $X_t^{(2)}$. As $p_t^{(1)} \in C_b^{\infty}(\mbb{R}^{d})$, cf.\ Remark~\ref{main-0}\eqref{main-0-i}, it follows easily from the differentiation lemma for parametrized integrals, see e.g.\ \cite{mims}, that \begin{equation*}
		\nabla p_t(x) = \int_{\mbb{R}^{d}} \nabla_x p_t^{(1)}(x-y) \, \mu_t(dy), \qquad x \in \mbb{R}^d, \,t>0.
	\end{equation*}
	Hence, by Tonelli's theorem and the gradient estimate \eqref{A2} for $p_t^{(1)}$, \begin{equation*}
		\int_{\mbb{R}^{d}} |\nabla p_t(x)| \, dx
		\leq \int_{\mbb{R}^{d}} \int_{\mbb{R}^{d}} |\nabla_x p_t^{(1)}(x-y)| \, dx \, \mu_t(dy)
		\leq c t^{-1/\alpha}, \qquad t \in (0,1).
	\end{equation*}
	\item It is obvious that the characteristic exponent $(\xi,\eta) \mapsto \psi^{(1)}(\xi) + \psi^{(2)}(\eta)$ of $(Z_t)_{t \geq 0}$ satisfies \eqref{A1} whenever $\psi^{(1)}$ and $\psi^{(2)}$ satisfy \eqref{A1}. Since the density of $Z_t$ is given by \begin{equation*}
		p_t(x,y) = p_t^{(1)}(x) p_t^{(2)}(y), \qquad x \in \mbb{R}^{d_1},\,y \in \mbb{R}^{d_2},\,t>0
	\end{equation*}
	we can follow the reasoning from the first part, i.e.\ apply the differentiation lemma and Tonelli's theorem, to find that \begin{equation*}
		\int_{\mbb{R}^{d_1+d_2}} |\nabla p_t(z)| \, dz
		\leq c t^{-1/\min\{\alpha_1,\alpha_2\}}, \qquad t \in (0,1). \qedhere
	\end{equation*}
\end{enumerate} \end{proof}

\begin{kor} \label{ex-3}
	Let $(X_t)_{t \geq 0}$ be a L\'evy process with L\'evy triplet $(b,Q,\nu)$. If $Q$ is positive definite, then the interior Schauder estimates in Theorem~\ref{main-1}, Corollary~\ref{main-3} and Corollary~\ref{main-5} hold with $\alpha=2$.
\end{kor}

If there is no jump part, i.e.\ $\nu=0$, then the infinitesimal generator associated with $(X_t)_{t \geq 0}$ is given by \begin{equation*}
	Af(x) = b \cdot \nabla f(x) + \frac{1}{2} \tr(Q \cdot \nabla^2 f(x)), \qquad f \in C_c^{\infty}(\mbb{R}^d),
\end{equation*}
and Corollary~\ref{ex-3} yields the classical interior Schauder estimates for solutions to the equation $Af=g$ associated with the second order differential operator $A$, see e.g.\ \cite{gilbarg}. 

\begin{proof}[Proof of Corollary~\ref{ex-3}]
	The L\'evy process $(X_t)_{t \geq 0}$ has a representation of the form \begin{equation*}
		X_t = Q W_t + J_t, \qquad t \geq 0,
	\end{equation*}
	where $(B_t)_{t \geq 0}$ is a Brownian motion and $(J_t)_{t \geq 0}$ is a L\'evy process with L\'evy triplet $(b,0,\nu)$. Since the transition density of $QB_t$ is of Gaussian type and $Q$ is positive definite, it is straightforward to check that $(QB_t)_{t \geq 0}$ satisfies \eqref{A1} and \eqref{A2} with $\alpha=2$. The L\'evy process $(J_t)_{t \geq 0}$ is independent of $(QB_t)_{t \geq 0}$, see e.g.\ \cite[Theorem II.6.3]{ikeda}, and therefore Proposition~\ref{ex-1}\eqref{ex-1-i} shows that $(X_t)_{t \geq 0}$ satisfies \eqref{A1} and \eqref{A2} with $\alpha=2$. If we choose $\gamma=2$, then $\int_{|y| \leq 1} |y|^{\gamma} \, \nu(dy)<\infty$ and the balance condition \eqref{A3} is trivial. Hence, the assumptions \eqref{A1}-\eqref{A3} of Theorem~\ref{main-1}, Corollary~\ref{main-3} and Corollary~\ref{main-5} are satisfied for $\alpha=\gamma=2$.
\end{proof}

Our next result applies to a large class of jump L\'evy processes, including stable L\'evy processes. It  is a direct consequence of the gradient estimates obtained in \cite{ssw12}.

\begin{kor} \label{ex-5}
	Let $(X_t)_{t \geq 0}$ be a pure-jump L\'evy process. Assume that its L\'evy measure $\nu$ satisfies \begin{align} \begin{aligned}
		\nu(B) &\geq \int_0^{r_0} \!\! \int_{\mbb{S}^{d-1}} \I_B(r \theta) r^{-1-\varrho} \, \mu(d\theta) \, dr \\ &\quad + \int_{r_0}^{\infty}\!\! \int_{\mbb{S}^{d-1}} \I_B(r\theta) r^{-1-\beta} \, \mu(d\theta) \, dr, \quad B \in \mc{B}(\mbb{R}^d \backslash \{0\}) \end{aligned} \label{ex-eq11}
	\end{align}
	for some constants $r_0>0$, $\varrho \in (0,2)$, $\beta \in (0,\infty]$ and a finite measure $\mu$ on the unit sphere $\mbb{S}^{d-1} \subseteq \mbb{R}^d$ which is non-degenerate, in the sense that its support is not contained in $\mbb{S}^{d-1} \cap V$ for some lower-dimensional subspace $V \subseteq \mbb{R}^d$. If \begin{equation}
		\int_{|y| \leq 1} |y|^{\gamma} \, \nu(dy) < \infty \label{ex-eq13}
	\end{equation}
	for some $\gamma < 1+\varrho$, then the interior Schauder estimates from Theorem~\ref{main-1}, Corollary~\ref{main-3} and Corollary~\ref{main-5} hold with $\alpha=\varrho$.
\end{kor}

If the L\'evy measure $\nu$ \emph{equals} the right-hand side of \eqref{ex-eq11}, then the assumption \eqref{ex-eq13} is trivially satisfied; this is, for instance, the case if $(X_t)_{t \geq 0}$ is isotropic $\alpha$-stable or relativistic $\alpha$-stable. In particular, Corollary~\ref{ex-5} generalizes \cite[Theorem 1.1]{ros-oton16}. \par \medskip
The next corollary gives a criterion for \eqref{A1}-\eqref{A3} in terms of the growth of the characteristic exponent of $(X_t)_{t \geq 0}$.

\begin{kor} \label{ex-6}
	Let $(X_t)_{t \geq 0}$ be a L\'evy process with infinitesimal generator $(A,\mc{D}(A))$. If the characteristic exponent $\psi$ satisfies the sector condition, $|\im \psi(\xi)| \leq c \re \psi(\xi)$, and \begin{equation}
		\re \psi(\xi) \asymp |\xi|^{\varrho} \quad \text{as $|\xi| \to \infty$} \label{ex-eq45}
	\end{equation}
	for some $\varrho \in (0,2)$, then the interior Schauder estimates in Theorem~\ref{main-1}, Corollary~\ref{main-3} and Corollary~\ref{main-5} hold with $\alpha=\varrho$.
\end{kor}

\begin{proof}
	The Hartman--Wintner condition \eqref{A1} is trivially satisfied. It follows from \cite{ssw12} that the gradient estimate $\int_{\mbb{R}^d} |\nabla p_t(x)| \, dx \leq ct^{-1/\varrho}$ holds for $t \in (0,1)$, i.e.\ $\alpha=\varrho$ in \eqref{A2}. Moreover, \eqref{ex-eq45} implies that the L\'evy measure $\nu$ satisfies $\int_{|y| \leq 1} |y|^{\beta} \, \nu(dy)<\infty$ for all $\beta>\varrho$, cf.\ \cite[Lemma A.2]{ihke}, and by choosing $\beta$ close to $\varrho$, we find that the balance condition \eqref{A3} holds. 
\end{proof}

Corollary~\ref{ex-6} covers many important and interesting examples, e.g.\  \begin{itemize}
	\item isotropic stable, relativistic stable and tempered stable L\'evy processes, 
	\item subordinated Brownian motions with characteristic exponent of the form $\psi(\xi) = f(|\xi|^2)$ for a Bernstein function $f$ satisfying $f(\lambda) \asymp \lambda^{\varrho/2}$ for large $\lambda$, cf.\ \cite{bernstein} for details.
	\item  L\'evy processes with symbol of the form \begin{equation*}
		\psi(\xi) = |\xi|^{\varrho} + |\xi|^{\beta}, \qquad \xi \in \mbb{R}^d,
	\end{equation*}
	for $\beta \in (0,\varrho)$. 
\end{itemize}

Our final example in this section is concerned with the operator \begin{equation*}
	Af(x,y) = - \left( - \frac{\partial^2}{\partial x^2} \right)^{\beta_1/2} f(x,y)   - \left( - \frac{\partial^2}{\partial y^2} \right)^{\beta_2/2} f(x,y), \qquad f \in C_c^{\infty}(\mbb{R}^2),\; x,y \in \mbb{R},
\end{equation*}
which arises as infinitesimal generator of a L\'evy process $(X_t)_{t \geq 0}$ of the form $X_t = (X_t^{(1)},X_t^{(2)})$ where $(X_t^{(i)})_{t \geq 0}$ are independent one-dimensional isotropic stable L\'evy processes with index $\beta_i \in (0,2]$, $i=1,2$, see e.g.\,\cite{pruitt69} for more information. The difference $|\beta_1-\beta_2|$ measures how much the behaviour of the first coordinate $(X_t^{(1)})_{t \geq 0}$ differs from the behaviour of the second coordinate  $(X_t^{(2)})_{t \geq 0}$. If $|\beta_1-\beta_2|$ is large (i.e.\ close to $2$), we are dealing with a highly anisotropic process. 

\begin{bsp} \label{ex-7}
	Let $(X_t)_{t \geq 0}$ be a two-dimensional L\'evy process with characteristic exponent \begin{equation*}
		\psi(\xi,\eta) = |\xi|^{\beta_1} + |\eta|^{\beta_2}, \qquad \xi, \eta \in \mbb{R}
	\end{equation*}
	for some constants $\beta_i \in (0,2]$, $i=1,2$. \begin{enumerate}
		\item The balance condition \eqref{A3} holds if, and only if, $|\beta_1-\beta_2|<1$.
		\item If $|\beta_1-\beta_2|<1$ then the interior Schauder estimates in Theorem~\ref{main-1}, Corollary~\ref{main-3} and Corollary~\ref{main-5} hold with $\alpha=\min\{\beta_1,\beta_2\}$.
	\end{enumerate}
\end{bsp}

Let us mention a further class of processes illustrating the role of the balance condition \eqref{A3}. Farkas \cite[Example 2.1.15]{farkas} showed that for every $0<\beta< \alpha<2$ there exists a one-dimensional L\'evy process whose characteristic exponent $\psi$ oscillates for $|\xi| \to \infty$ between $|\xi|^{\beta}$ and $2|\xi|^{\alpha}$. Since the growth of $\psi$ at infinity is closely linked to existence of fractional moments $\int_{|y| \leq 1} |y|^{\gamma} \, \nu(dy)$, it stands to reason that $\int_{|y| \leq 1} |y|^{\gamma} \, \nu(dy)<\infty$ only for $\gamma>\beta$ and $\int_{\mbb{R}^d} |\nabla p_t(x)| \, dx \leq Mt^{-1/\alpha}$. In particular, the balance condition \eqref{A3} fails if $\alpha-\beta>1$.

\appendix

\section{}

For the proof of our results we used the following lemma, which was already stated in Section~\ref{def}.

\begin{lem} \label{app-1}
	Let $\alpha \in (0,\infty)$, and let $U \subseteq \mbb{R}^d$ be open. The following statements hold for any $j \geq k := \lfloor \alpha \rfloor+1$: \begin{enumerate}
		\item\label{app-1-i} There exists a constant $c>0$ such that \begin{equation}
			\sup_{0<|h| \leq r} \frac{|\Delta_h^k f(x)|}{|h|^{\alpha}} 
			\leq c r^{-\alpha} \|f\|_{\infty,U} + c \sup_{0<|h| \leq r/j} \sup_{z \in B(x,r(k+1))} \frac{|\Delta_h^j f(z)|}{|h|^{\alpha}} \label{app-eq3}
		\end{equation}
		for all $f \in C_b(U)$, $r >0$, and $x \in U$ with $B(x,r(k+2)) \subseteq U$.
		\item\label{app-1-ii} If $\alpha>1$ then there exists a constant $c>0$ such that \begin{equation}
			\max_{i=1,\ldots,d} |\partial_{x_i} f(x)|
			\leq c r^{-\alpha} \|f\|_{\infty,U} + c \sup_{0<|h| \leq r/j} \sup_{z \in B(x,r(k+1))} \frac{|\Delta_h^j f(z)|}{|h|^{\alpha}} \label{app-eq4}
		\end{equation}
		for all $f \in C_b^1(U)$, $r >0$ and $x \in U$ with $B(x,r(k+2)) \subseteq U$.
	\end{enumerate}
\end{lem}

\begin{proof}
	First of all, we note that it suffices to prove both statements for $f \in C_b^2(U)$; the inequalities can be extended using a standard approximation technique, i.e.\ by considering $f_i := f \ast \varphi_i$ for a sequence of mollifiers $(\varphi_i)_{i \geq 1}$. \par
	Denote by $\tau_h f(x) := f(x+h)$ the shift operator. A straight-forward computation shows that  \begin{equation}
		\Delta_h^n(u \cdot v) = \sum_{\ell=0}^n {n \choose \ell} \Delta_h^\ell u \cdot \Delta_h^{n-\ell} \tau_h^\ell v \label{app-eq5}
	\end{equation}
	holds for any $n \in \mbb{N}$, $h \in \mbb{R}^d$ and any two functions $u,v$. \par
	To prove \eqref{app-1-i} we note that the assertion is obvious for $j=k$, and so it suffices to consider $j>k$. We will first establish the following auxiliary statement: There exists a constant $C>0$ such that \begin{equation}
		\sup_{0<|t| \leq 1} \frac{|\Delta_t^k g(0)|}{|t|^{\alpha}} \leq C \|g\|_{\infty,(-k-2,k+2)} + C \sup_{0<|t| \leq 1/j} \sup_{|y| \leq k+1} \frac{|\Delta_t^j g(y)|}{|t|^{\alpha}} \label{app-eq7}
	\end{equation}
	for any twice differentiable bounded function $g: (-k-2,k+2) \to \mbb{R}$. To this end, pick $\chi \in C_c^{\infty}(\mbb{R})$ such that $\I_{[-k,k]} \leq \chi \leq \I_{(-k-1/3,k+1/3)}$. Clearly, \begin{equation*}
		\sup_{0<|t| \leq 1} \frac{|\Delta_t^k g(0)|}{|t|^{\alpha}} = \sup_{0<|t| \leq 1} \frac{|\Delta_t^k (g \chi)(0)|}{|t|^{\alpha}}.
	\end{equation*}
	Using the equivalence of the norms on $\mc{C}_b^{\alpha}(\mbb{R})$, cf.\ \eqref{def-eq5}, we get  \begin{align*}
		\sup_{0<|t| \leq 1} \frac{|\Delta_t^k g(0)|}{|t|^{\alpha}} 
		\leq \|g \chi\|_{\mc{C}_b^{\alpha}(\mbb{R})}
		&\leq c_1 \|g\chi\|_{\infty} + c_1 \sup_{|t| \leq 1} \sup_{y \in \mbb{R}} \frac{|\Delta_t^{2j} (g \chi)(y)|}{|t|^{\alpha}} \\
		&\leq c_1' \|g\|_{\infty,(-k-1,k+1)} + c_1 \sup_{|t| \leq 1/(6j)} \sup_{y \in \mbb{R}} \frac{|\Delta_t^{2j} (g \chi)(y)|}{|t|^{\alpha}} 
	\end{align*}
	for some constants $c_1$ and $c_1'$. As $\chi=0$ on $\mbb{R} \backslash (-k-1/3,k+1/3)$, we have $\Delta_t^{2j} g(y)=0$ for all $|y|>k+\frac{2}{3}$ and $|t| \leq \frac{1}{6j}$. Consequently, 
	\begin{align}
		\sup_{0<|t| \leq 1} \frac{|\Delta_t^k g(0)|}{|t|^{\alpha}} 
		&\leq c_1'  \|g\|_{\infty,(-k-1,k+1)} + c_1 \sup_{|t| \leq 1/(6j)} \sup_{|y| \leq k+2/3} \frac{|\Delta_t^{2j} (g \chi)(y)|}{|t|^{\alpha}}. \label{app-eq9}
	\end{align}
	Since \begin{equation*}
		|\Delta_t^{\ell} g(y)|
		= |\Delta_t^{\ell-j} \Delta_t^j g(y)|
		\leq 2^{\ell-j} \sup_{|z-y| \leq (\ell-j)t} |\Delta_t^{j} g(z)|, \qquad \ell \geq j,
	\end{equation*}
	and \begin{equation*}
		\|\Delta_t^{\ell} \chi\|_{\infty}
		\leq c |t|^k \|\chi\|_{C_b^k(\mbb{R})}, \qquad \ell \geq j \geq k,
	\end{equation*}
	an application of the product formula \eqref{app-eq5} gives \begin{align*}
		|\Delta_t^{2j} (g \chi)(y)|
		&\leq c_2  \|g\|_{\infty,(-k-1,k+1)} \sum_{\ell=0}^{j-1} \|\Delta_t^{2j-\ell} \chi\|_{\infty}+ c_2 \|\chi\|_{\infty} \sum_{\ell=j}^{2j} |\Delta_t^{\ell} g(y)| 	\\
		&\leq c_3 \|g\|_{\infty,(-k-1,k+1)} |t|^k \|\chi\|_{C_b^k(\mbb{R})} + c_3 \sup_{|z-y| \leq 2tj} |\Delta_t^j g(z)|
	\end{align*}
	for all $|y| \leq k+\frac{2}{3}$ and $|t| \leq \frac{1}{6j}$. Combining this estimate with \eqref{app-eq9} and noting that $\alpha \leq k$ proves \eqref{app-eq7}. Now if $f \in C_b^2(U)$, then we apply \eqref{app-eq7} with $g(t) := f(x+rt h)$ for fixed $|h|=1$ to get the desired inequality.  \par
	It remains to prove \eqref{app-1-ii}. First we consider the case $\alpha \in (1,2)$ and $j=2$. The auxiliary inequality which we need is \begin{equation}
		|g'(0)| \leq C \|g\|_{\infty,(-3,3)} + C \sup_{|t| \leq 1/2} \sup_{|y| \leq 2} \frac{|\Delta_t^2 g(y)|}{|t|^{\alpha}} \label{app-eq15}
	\end{equation}
	for a uniform constant $C>0$ where $g: \mbb{R} \to \mbb{R}$ is differentiable on $(-3,3)$. To this end, choose $\chi \in C_c^2(\mbb{R})$ with $\I_{[-1,1]} \leq \chi \leq \I_{(-4/3,4/3)}$. By the equivalence of the norms on the H\"{o}lder--Zygmund space $\mc{C}_b^{\alpha}(\mbb{R})$, cf.\ \eqref{def-eq5}, we get \begin{equation*}
		|g'(0)| = |(g \chi)'(0)| 
		\leq \|g \chi \|_{C_b^1(\mbb{R})} 
		\leq \|g \chi\|_{\mc{C}_b^{\alpha}(\mbb{R})} 
		\leq c_4 \|g \chi\|_{\infty}  + c_4 \sup_{|t| \leq 1} \sup_{y \in \mbb{R}} \frac{|\Delta_t^4 (g \chi)(y)|}{|t|^{\alpha}}
	\end{equation*}
	for some finite constant $c_4>0$. Following the reasoning in the first part of the proof (with $k=1$ and $j=2$) yields \eqref{app-eq15}. Applying \eqref{app-eq15} for $g(t) := f(x+rt e_j)$, where $e_j$ is the $j$-th unit vector in $\mbb{R}^d$, gives \eqref{app-1-ii} for $\alpha \in (1,2)$ and $j=2$. In combination with \eqref{app-1-i}, this yields the desired inequality for every $\alpha>1$ and $j \geq \floor{\alpha}+1$.
\end{proof}

\end{document}